\documentclass[10pt,a4paper]{article}

\usepackage{booktabs} 
\usepackage{array}    

\usepackage[below]{placeins} 
\usepackage{sidecap}
\usepackage{graphicx}
\usepackage{moreverb}   
\usepackage{hyperref}
\usepackage{bbm}
\usepackage{amssymb}
\usepackage{xcolor}

\usepackage{fancyhdr} 

\usepackage[latin1]{inputenc}
\usepackage{wrapfig}
\usepackage[numbers]{natbib} 
\usepackage{enumitem}
\usepackage{amsmath}
\usepackage{amsthm}
\newtheorem{thm}[subsection]{Theorem}
\newtheorem{lem}[subsection]{Lemma}
\newtheorem{cor}[subsection]{Corollary}

\newtheorem{pro}[subsection]{Proposition}
\newtheorem{exa}[subsection]{Example}
\newtheorem{defn}[subsection]{Definition}
\newtheorem{rem}[subsection]{Remark}

\newcommand{\C}{\mathbb{C}}

\newcommand{\Z}{\mathbb{Z}}
\newcommand{\N}{\mathbb{N}}

\newcommand{\norm}[1]{\left\lVert #1 \right\rVert}

\numberwithin{equation}{section}

\author{\small{ \textsc{SOLANGE BRIDGITTE} DIFO} \\ \small {African Institute for Mathematical Sciences Cameroon} \\ \small{University of Yaounde I}}
\title{\small{ XIA'S THEOREM FOR THE FOCK SPACE} \(H^{2}(\mathbb{C}^n, d\mu)\)}

\begin{document}
	
	\maketitle

	\begin{abstract}
		In this paper, we provide a detailed proof for Xia's following theorem: the \(C^{*}\)-algebra generated by the class of weakly localized operators on \(H^{2}(\C^n,d\mu)\) coincides with \(\mathcal{T}^{(1)}\).
	\end{abstract}
	
	\section{Introduction}
	
	For some \(\alpha>0\), \(p>0\) and \(dV\) the standard volume measure on \(\C^n\), let \(L^{p}_{\alpha}(\C^n,dV)\) be the Lebesgue space of measurable functions \(f\) on \(\C^n\) such that 
	\begin{equation}\label{Fnorm1}
		\norm{f}^{p}_{\alpha} := \left(\frac{p\alpha}{2\pi}\right)^{n} \int_{\C^n} |f(z)e^{-\frac{\alpha}{2}|z|^2}|^{p} dV(z) < \infty~.
	\end{equation}
	Similarly, for \(\alpha>0\) and \(p=\infty\), we denoted by \(L^{\infty}_{\alpha}\) the space of Lebesgue measurable function \(f\) on \(\C^n\) such that 
	\begin{equation*}
		\norm{f}_{\infty,\alpha} := \text{esssup}\{|f(z)|e^{-\frac{\alpha|z|^2}{2}} : ~z \in \C^n\}< \infty ~.
	\end{equation*}
	The classical Fock space \(F^{p}_{\alpha}\) is the space of entire functions on \(\C^n\) which belong to \( L^{p}_{\alpha}(\C^n, dV)\). Similarly, the Fock space \(F^{\infty}_{\alpha}\) is the space of entire functions on \(\C^n\) which belong to \(L^{\infty}_{\alpha}\).
	
	Let \(d\mu\) be the Gaussian measure on \(\C^n\), \(n\ge 1\). In terms of the standard volume measure \(dV\) on \(\C^n\), it is given by
	\begin{equation*}
		d\mu(z) = \pi^{-n} e^{-|z|^2} dV(z)~.
	\end{equation*}
	
	 The Fock space \(H^{2}(\C^n,d\mu)\) is defined to be the subspace of the (Hilbert-) Lebesgue space \(L^{2}(\C^n,d\mu)\) consisting of entire functions. Notice that \(H^{2}(\C^n,d\mu) = F^{2}_{1} \). The symbol \(K_z\) denotes the reproducing kernel and the symbol \(k_z\) denotes the normalized reproducing kernel for \(H^{2}(\C^n,d\mu)\). That is, 
	\begin{equation*}
		K_z(\zeta) = e^{\langle \zeta, z \rangle}, \quad k_{z}(\zeta) = e^{\langle \zeta, z\rangle}e^{-\frac{|z|^2}{2}}~, \quad z,\zeta \in \C^n.
	\end{equation*}

	In \cite{xia2015localization}, J. Xia showed in the case of the Bergman space on the unit ball of \(\C^n\) that the norm closure of \(\{T_{f}: f \in L^{\infty}(B,dv) \}\) coincides with the \(C^{*}\)-algebra of weakly localized operators. Also, he stated in \cite[Section 4 ]{xia2015localization}  that  the analogue of \cite[Theorem 1.5]{xia2015localization} on the Fock space \(H^{2}(\C^n, d\mu)\) was true. In this paper, we define the notion of weakly localized operators, state Xia's theorem for the Fock space \(H^{2}(\C^n, d\mu)\) and provide details of its proof. Further, we present a consequence of this theorem on the compactness of operators on \(H^{2}(\C^n,d\mu)\).  We begin with the following definitions and we state the main theorem, the proof of which will retain our attention in the following sections.

	\begin{defn}
		For \(f \in L^{\infty}(\C^n,dV) \), the \textbf{Toeplitz operator} \(T_{f}\) is defined by the formula
		\begin{equation*}
			T_{f}h = P(fh)~,\quad h \in H^{2}(\C^n,d\mu)~,
		\end{equation*}
		where \(P \colon L^{2}(\C^n,d\mu) \rightarrow H^{2}(\C^n,d\mu)\) is the orthogonal projection.
	\end{defn}
	
	The \textbf{standard lattice} in \(\C^n\) is denoted by 
	\begin{equation*}
		\Z^{2n} = \{(m_{1}+il_{1}, \dots, m_{n}+il_{n}): m_{1}, l_{1}, \dots, m_{n},l_{n} \in \Z\}~.
	\end{equation*}
	We fix an orthonormal set \(\{e_{u}: u \in \Z^{2n}\}\) in \(H^{2}(\C^n,d\mu)\). We let \(S\) denote the \textbf{fundamental unit cube in }\(\C^{n}\). That is,
	\begin{equation*}
		S = \{(x_{1}+iy_{1}, \dots, x_{n}+iy_{n}) : x_{1},y_{1}, \dots , x_{n},y_{n} \in [0,1)\}.
	\end{equation*}
	With \(\Z^{2n}\) and \(S\), we have 
	\begin{equation*}
		\cup_{u \in \Z^{2n}} \{S+u\} = \C^{n} = \cup_{u \in \Z^{2n}} \{u-S\} ~,
	\end{equation*}
	which is a tiling of the space, meaning that there is no overlap between \(S+u\) and \(S+v\) for \(u \neq v\) in \(\Z^{2n}\) (resp. between \(u-S\) and \(v-S\) for \(u \neq v \in \Z^{2n}\) ).

	\begin{defn}
		Let \(\mathcal{T}^{(1)}\) denote the norm closure of \(\{T_{f}: f\in L^{\infty}(\C^n,dV)\}\) in \(\mathcal{B}(H^2(\C^n,d\mu))\) with respect to the operator norm. That is 
		\begin{equation*}
			\mathcal{T}^{(1)} = \{B : \lim_{k \to \infty } \norm{B-T_{b_{k}}}=0 , b_{k} \in L^{\infty}(\C^{n},dV)\}.
		\end{equation*}
	\end{defn}
	
	\begin{defn}
		We denote by \(\mathcal{S}\) the linear span of the normalized reproducing kernels \(k_z\). A linear operator \(B : \mathcal{S} \rightarrow H^{2}(\C^n,d\mu) \) is said to be \textbf{admissible} on \( H^{2}(\C^n,d\mu) \) if there exists a linear operator \( B^{*} : \mathcal{S} \rightarrow H^{2}(\C^n,d\mu) \) such that the duality relation 
		\begin{equation}\label{R1}
			\langle Bk_{z}, k_{w} \rangle = \langle k_{z}, B^{*}k_w \rangle
		\end{equation}
		holds for all \(z,w \in \C^n\).	
	\end{defn}
	The inner product here is with respect to \(d\mu\).
	
	We define below sufficiently localized operators following J. Xia and D. Zheng (XZ) in \cite{xia2013localization}.
	\begin{defn}
		A bounded linear operator \(B\) on \(H^{2}(\C^n,d\mu)\) is said to be XZ-\textbf{sufficiently localized} if there exist constants \(2n<\beta<\infty\) and \(0<C<\infty\) such that
		\begin{equation}\label{SL1}
			|\langle Bk_{z},k_{w}\rangle| \le \frac{C}{(1+|z-w|)^{\beta}}
		\end{equation}
		for all \(z,w\in \C^n\).
	\end{defn}

		\begin{defn}
		An admissible operator \(B\) on \(H^{2}(\C^n,d\mu)\) is said to be \textbf{weakly localized} if it satisfies the following four conditions
		\begin{equation*} 
			\sup_{z\in \C^n}\int_{\C^n} |\langle Bk_{z}, k_{w} \rangle| dV(w) <\infty, \quad \sup_{z\in \C^n}\int_{\C^n} |\langle B^{*} k_{z}, k_{w} \rangle| dV(w) <\infty
		\end{equation*}
		
		and 
		\begin{equation*}
			\lim_{r\to\infty}\sup_{z\in \C^n}\int_{|z-w|\ge r} |\langle Bk_{z}, k_{w} \rangle| dV(w) =0, \quad \lim_{r\to\infty}\sup_{z\in \C^n}\int_{|z-w|\ge r} |\langle B^{*} k_{z}, k_{w} \rangle| dV(w) =0~.
		\end{equation*}
	\end{defn}
	
	\begin{exa}\label{Ex1}
		If \(f\) is a \textbf{bounded measurable function} on \(\C^n\), then there is a positive constant \(0<C=C(f)<\infty\) such that 
		\begin{equation*}
			|\langle T_{f}k_{z}, k_{w}\rangle| \le C e^{-(1/8) |z-w|^2}
		\end{equation*}
		for all \(z,w \in \C^n\). That is, \(T_{f} \in \text{XZ-}\mathcal{SL}\). Also, \(T_{f} \in \mathcal{L}_{p}\) for \(2<p<\frac{8}{3}\).
	\end{exa}
	
	\begin{proof}
		For each \(z, w, \xi \in \C^n\), we have 
		\begin{equation*}
			|k_{z}(\xi)k_{w}(\xi)|e^{-|\xi|^2} = |e^{\langle \xi , z \rangle - \frac{|z|^2}{2} } e^{\langle \xi, w \rangle - \frac{|w|^2}{2}} | e^{-|\xi|^2} = e^{-\frac{1}{2}\left(|z-\xi|^2 + |w-\xi|^2\right)} ~.
		\end{equation*}
		By the triangular inequality we have 
		\begin{equation*}
			|z-w|^2 \le |z-\xi + \xi -w|^2 \le (|z-\xi| + |w-\xi|)^2 \le 2 (|z-\xi|^2 + |w-\xi|^2)~.
		\end{equation*}
		It follows that
		\begin{eqnarray}\label{kzw}
			|k_{z}(\xi)k_{w}(\xi)|e^{-|\xi|^2} = e^{-\frac{1}{4}\left(|z-\xi|^2 + |w-\xi|^2\right)}e^{-\frac{1}{4}\left(|z-\xi|^2 + |w-\xi|^2\right)} \le e^{-\frac{1}{4}\left(|z-\xi|^2 + |w-\xi|^2\right)} e^{-\frac{1}{8}|z-w|^2}~.
		\end{eqnarray}
		Let \(f\) be a bounded measurable function on \(\C^n\), then for all \(g,h\in H^{2}(\C^n,d\mu)\), it holds
		\begin{equation}\label{Toepfgh}
			\langle T_{f}g, h \rangle = \int_{\C^{n}} f(\xi)g(\xi)\overline{h(\xi)} d\mu(\xi)~.
		\end{equation}
		In fact, 
		\begin{eqnarray*}
			\langle T_{f}g, h \rangle &=&  \int_{\C^{n}} T_{f}g(w)\overline{h(w)} d\mu(w) \\
			&=& \int_{\C^{n}}\int_{\C^{n}} K(w,\xi) f(\xi)g(\xi) d\mu(\xi) \overline{h(w)} d\mu(w) \\
			&=& \int_{\C^{n}} f(\xi)g(\xi) \int_{\C^{n}} K_{\xi}(w) \overline{h(w)} d\mu(w) d\mu(\xi) \\
			&=& \int_{\C^{n}} f(\xi)g(\xi) \langle K_{\xi} , h \rangle d\mu(\xi) = \int_{\C^{n}} f(\xi)g(\xi) \overline{\langle  h, K_{\xi}  \rangle} d\mu(\xi) \\
			&=& \int_{\C^{n}} f(\xi)g(\xi)\overline{h(\xi)} d\mu(\xi)~.
		\end{eqnarray*}
		Therefore,
		\begin{eqnarray*}
			|\langle T_{f}k_{z}, k_{w} \rangle| &=& \left|\int_{\C^n} 
			f(\xi)k_{z}(\xi) \overline{k_{w}(\xi)} d\mu(\xi)\right|  \\
			&\le & \frac{\norm{f}_{\infty}}{\pi^n} \int_{\C^n} |k_{z}(\xi) k_{w}(\xi)| e^{-|\xi|^2} dV(\xi) \\
			& \le & \frac{\norm{f}_{\infty}}{\pi^n}  e^{-\frac{1}{8}|z-w|^2} \int_{\C^n} e^{-\frac{1}{4}\left(|z-\xi|^2 + |w-\xi|^2\right)}  dV(\xi) \\
			& \le & \frac{\norm{f}_{\infty}}{\pi^n}  e^{-\frac{1}{8}|z-w|^2} \left(\int_{\C^n} e^{-\frac{1}{2}|z-\xi|^2} dV(\xi) \right)^{\frac{1}{2}} \left(\int_{\C^n} e^{-\frac{1}{2}|w-\xi|^2}dV(\xi) \right)^{\frac{1}{2}} \\
			&=& \frac{\norm{f}_{\infty}}{\pi^n}  e^{-\frac{1}{8}|z-w|^2} \int_{\C^n} e^{-\frac{1}{2}|\zeta|^2} dV(\zeta) \\
			&=& (\sqrt{2})^n \norm{f}_{\infty} e^{-\frac{1}{8}|z-w|^2}~.
		\end{eqnarray*}	
		This completes the proof.
	\end{proof}
	
	 \begin{pro}
	 	Any XZ-sufficiently localized operator is weakly localized.
	 \end{pro} 
	\begin{proof}
	Let \(B\) be a XZ-sufficiently localized operator, then  for all \(z\in \C^n\), we have
	\begin{eqnarray*}
		\int_{\C^n} |\langle Bk_z, k_w \rangle| dV(w) &\le& \frac{C}{\pi^n} \int_{\C^n} \frac{1}{(1+|\zeta|)^{\beta}} dV(\zeta) \\
		&=& \frac{Cc_n}{\pi^n} \int_{0}^{\infty} \frac{r^{2n-1}}{(1+r)^{\beta}} dr \\
		&=& C_n \left[ \int_{0}^{1} \frac{r^{2n-1}}{(1+r)^{\beta}} dr  + \int_{1}^{\infty } \frac{r^{2n-1}}{(1+r)^{\beta}} dr\right] \\
		&\le & C_n \left[ \int_{0}^{1} \frac{1}{(1+r)^{\beta}} dr +  \int_{1}^{\infty } \frac{r^{2n-1}}{r^{\beta}} dr\right] \\
		&=& C_n \left\{\frac{1}{\beta-1}\left(1-\frac{1}{2^{\beta}-1}\right) + \frac{1}{\beta-2n}\right\} \\
		&:=& C(n,\beta)~,
	\end{eqnarray*}
	which does not depend on \(z\). Therefore,
	\begin{equation*}
		\sup_{z\in \C^n}\int_{\C^n} |\langle Bk_z, k_w \rangle| dV(w) \le C(n,\beta)  <\infty~.
	\end{equation*}
	And in the same way, we also have
	\begin{equation*}
		\sup_{z\in \C^n}\int_{\C^n} |\langle B^*k_z, k_w \rangle| dV(w) \le C(n,\beta) <\infty~.
	\end{equation*}
	We now show that \(B\) satisfies the third condition of weakly localized operator.
	
	Let \(z\in \C^n\), using the change of variables \(\zeta=z-w\) and the spherical coordinates, it holds
	\begin{eqnarray*}
		\int_{|z-w|\ge r} |\langle Bk_z, k_w \rangle dV(w) &\le& C \int_{|z-w|\ge r} \frac{1}{(1+|z-w|)^{\beta}} dV(w) \\
		&=& C \int_{|\zeta|\ge r} \frac{1}{(1+|\zeta|)^{\beta}} dV(\zeta) \\
		&=& C c_n \int_{\rho\ge r} \frac{\rho^{2n-1}}{(1+\rho)^{\beta}} d \rho \\
		&\le& C c_n \int_{\rho\ge r} \frac{\rho^{2n-1}}{\rho^{\beta}} d \rho =  C c_n \int_{\rho\ge r} \frac{1}{\rho^{\beta-2n+1}} d \rho \\
		&=& \frac{Cc_n}{\beta-2n}\frac{1}{r^{\beta-2n}}~.
	\end{eqnarray*}
	It follows that,
	\begin{equation*}
		\lim_{r\to\infty} \sup_{z\in \C^n} \int_{|z-w|\ge r} |\langle Bk_z, k_w \rangle dV(w) \le \frac{Cc_n}{\beta-2n} \lim_{r\to\infty} \frac{1}{r^{\beta-2n}} =0~.
	\end{equation*}
	Using the duality relation (\ref{R1}) of an admissible operator and similarly like with \(B\), we also have:
	\begin{equation*}
		\lim_{r\to\infty} \sup_{z\in \C^n} \int_{|z-w|\ge r} |\langle B^*k_z, k_w \rangle dV(w) \le \frac{Cc_n}{\beta-2n} \lim_{r\to\infty} \frac{1}{r^{\beta-2n}} =0~.
	\end{equation*}
\end{proof}
	
	We denote by \(\mathcal{WL}\) the collection of weakly localized operators on \(H^{2}(\C^n,d\mu)\). We recall the following definition from \cite{zhu2007}.
	
		\begin{defn}
		A \textbf{Banach algebra} is a complex algebra together with a complete norm satisfying the condition \(\norm{xy} \le \norm{x}\norm{y}\).
		A \(C^{*}\)\textbf{-algebra} is a Banach algebra \(\mathcal{A}\) with an involution \(x \mapsto x^{*}\) on it satisfying the following conditions:
		\begin{enumerate}[label=$\blacktriangleright$]
			\item \(x^{**}=x\) for all vectors \(x \in \mathcal{A} \).
			\item \((ax + by)^{*} = \bar{a}x^{*} + \bar{b}y^{*} \) for all vectors \(x,y \in \mathcal{A}\) and \(a,b \in \C\).
			\item \((xy)^{*}=y^{*}x^{*}\) for all vectors \(x,y \in \mathcal{A}\).
			\item \(\norm{xx^{*}} = \norm{x}^{2}\) for all vector \(x \in \mathcal{A}\).
		\end{enumerate}
	\end{defn}
	
	\begin{defn}
		A Banach algebra \(\mathcal{A}\) is called a \textbf{\(\star\)-algebra} if for every \(A\in \mathcal{A}\), we have \(A^{*} \in \mathcal{A} \).
	\end{defn}
	
	\begin{defn}
		We denote by \(C^{*}(\mathcal{WL})\) the \(C^{*}\)-algebra generated by weakly localized operators on \(H^{2}(\C^n, d\mu)\). Also, \(C^{*}(\mathcal{WL})\) is actually the norm closure of \(\mathcal{WL}\) since \(\mathcal{WL} \) is a \(\star\)-algebra.
	\end{defn}
	
	We will prove the following main result.
	\begin{thm}\label{XiaFock}
		We have 
		\begin{equation*}
			C^{*} (\mathcal{WL})=\mathcal{T}^{(1)}.
		\end{equation*}
	\end{thm}
	
	The organization of this paper is as follows. In Section 2, we will give propositions in the case of the Fock space \(H^{2}(\C^n, d\mu)\) which are the analogue of those given by Xia in \cite{xia2015localization} in the Bergman space case of the unit ball. Later, using these propositions in Section 3, we establish the proof of Theorem \ref{XiaFock} and present a consequence.

	\section{Preliminary Results}
	
	In this section, we will present results that will be used to establish the proof of Theorem \ref{XiaFock}. 
	
	\begin{pro}
		The set \(\mathcal{WL}\) is a \(\star\)-algebra.
	\end{pro}
	
	\begin{proof}
		From the definition of weakly localized operator on \(H^{2}(\C^n, d\mu)\), we know that if \(B \in \mathcal{WL}\), then also is \(B^{*}\). Moreover every linear combination of two operators in \(\mathcal{WL}\) is also in \(\mathcal{WL}\). Therefore, to complete the proof, we just have to prove that if \(B_{1}, B_{2} \in \mathcal{WL} \), then \(B_{1}B_{2} \in \mathcal{WL} \). 
		
		Let \(B_{1}, B_{2} \in \mathcal{WL}\), we denote indistinguishably by \(C\) the constant satisfying:
		\begin{equation*}
			\sup_{z\in \C^n} \int_{\C^{n}} |\langle B_{j}k_{z}, k_{w} \rangle | dV(w) < C \quad   \text{ and } \quad \sup_{z\in \C^n} \int_{\C^{n}} |\langle B^{*}_{j}k_{z}, k_{w} \rangle | dV(w) < C,
		\end{equation*}
		for \(j=1,2\). Let \(z\in \C^n \), we have
		\begin{eqnarray*}
			\int_{\C^{n}}| \langle B_{1} B_{2} k_{z}, k_{w} \rangle | dV(w) &=&  \int_{\C^{n}} |\langle B_{2} k_{z} , B^{*}_{1}k_{w} \rangle | dV(w) = \int_{\C^{n}} \left|\int_{\C^{n}} B_{2}k_{z} (\xi) \overline{B^{*}_{1}k_{w}(\xi) } d\mu(\xi) \right| dV(w) \\
			&= & \frac{1}{\pi^n} \int_{\C^{n}} \left|\int_{\C^{n}} \langle B_{2}k_{z} , k_{\xi} \rangle \langle k_{\xi}, B^{*}_{1} k_{w} \rangle  dV(\xi) \right| dV(w) \\
			&\le & \frac{1}{\pi^n} \int_{\C^{n}} \left(\sup_{\xi\in \C^n }\int_{\C^{n}} |\langle B_{1}k_{\xi}, k_{w}\rangle | dV(w)\right)  |\langle B_{2}k_{z}, k_{\xi} \rangle | dV(\xi) \\
			&<& C   \int_{\C^{n}} |\langle B_{2}k_{z}, k_{\xi} \rangle | dV(\xi)~.
		\end{eqnarray*}
		Hence, we have 
		\begin{equation*}
			\sup_{z\in \C^n} \int_{\C^{n}}| \langle B_{1} B_{2} k_{z}, k_{w} \rangle | dV(w) < C \sup_{z\in \C^n} \int_{\C^{n}} |\langle B_{2}k_{z}, k_{\xi} \rangle | dV(\xi) < C^{2} < \infty~.
		\end{equation*}
		We also have
		\begin{eqnarray}\label{weakproof}
			\int_{|z-w|\ge r} | \langle B_{1} B_{2} k_{z}, k_{w} \rangle | dV(w) &\le & \frac{1}{\pi^n} \int_{|z-w|\ge r} \int_{\C^{n}} |\langle B_{2}k_{z} , k_{\xi} \rangle| |\langle k_{\xi}, B^{*}_{1} k_{w} \rangle|  dV(\xi) dV(w) \nonumber \\
			&=& \frac{1}{\pi^n} \int_{\C^{n}} \left( \int_{|z-w|\ge r} |\langle B_{2}k_{z} , k_{\xi} \rangle| |\langle B_{1} k_{\xi},  k_{w} \rangle| dV(w) \right) dV(\xi) \nonumber \\
			&=& \frac{1}{\pi^n} \int_{|z-\xi|< \frac{r}{2}} I_{z}(\xi) dV(\xi) + \frac{1}{\pi^n} \int_{|z-\xi|\ge \frac{r}{2}} I_{z}(\xi) dV(\xi)~,
		\end{eqnarray}
		where \( I_{z} (\xi) = \int_{|z-w|\ge r} |\langle B_{2}k_{z} , k_{\xi} \rangle| |\langle B_{1}k_{\xi}, k_{w} \rangle| dV(w)\) . 
		
		For \(\xi \in B(z,\frac{r}{2}) \), we have \(B(\xi, \frac{r}{2}) \subset B(z,r) \) and hence, \(B(z,r)^c \subset B(\xi,\frac{r}{2})^{c} \). Therefore we can dominate the first integral as follows
		\begin{eqnarray*}
			\int_{|z-\xi|< \frac{r}{2}} I_{z}(\xi) dV(\xi)  &\le & \int_{|z-\xi|< \frac{r}{2}} \int_{|\xi-w| \ge \frac{r}{2}} |\langle B_{2}k_{z} , k_{\xi} \rangle| |\langle k_{\xi}, B^{*}_{1} k_{w} \rangle| dV(w)dV(\xi) \\
			&\le & \int_{|z-\xi|< \frac{r}{2}} \left( \sup_{\xi\in \C^n }\int_{|\xi-w| \ge \frac{r}{2}}  |\langle B_{1}k_{\xi}, k_{w} \rangle| dV(w)  \right) |\langle B_{2}k_{z} , k_{\xi} \rangle| dV(\xi) \\
			&=& C\left(\frac{r}{2}\right) \int_{|z-\xi|< \frac{r}{2}} |\langle B_{2}k_{z} , k_{\xi} \rangle| dV(\xi) \le C\left(\frac{r}{2}\right) \int_{\C^n} |\langle B_{2}k_{z} , k_{\xi} \rangle| dV(\xi) ~,
		\end{eqnarray*}
		where \( C(r) = \sup_{\xi\in \C^n}\int_{|\xi-w|\ge r} |\langle B_{2}k_{\xi}, k_{w} \rangle| dV(w)  \). Taking the supremum on \(z \in \C^n\), we have
		\begin{equation*}
			\sup_{z\in \C^n} \int_{|z-\xi|< \frac{r}{2}} I_{z}(\xi) dV(\xi) \le C\left(\frac{r}{2}\right) \sup_{z\in \C^n} \int_{\C^n} |\langle B_{2}k_{z} , k_{\xi} \rangle| dV(\xi) <C~ C\left(\frac{r}{2}\right)~,
		\end{equation*}
		which tends to \(0\) as \(r\) goes to \(\infty\) from the third property of weakly localized operators.
		
		On the other hand, the second integral in relation (\ref{weakproof}) can be dominated as follows
		\begin{eqnarray*}
			\int_{|z-\xi|\ge \frac{r}{2}} I_{z}(\xi) dV(\xi) &\le&   \int_{|z-\xi|\ge\frac{r}{2} } \int_{\C^n}  |\langle B_{1} k_{\xi},  k_{w} \rangle|  dV(w)~ |\langle B_{2}k_{z} , k_{\xi} \rangle| dV(\xi) \\
			&\le & \int_{|z-\xi|\ge\frac{r}{2} } \left( \sup_{\xi \in \C^n} |\langle B_{1} k_{\xi},  k_{w} \rangle|  dV(w) \right)|\langle B_{2}k_{z} , k_{\xi} \rangle| dV(\xi) \\
			&<& C\int_{|z-\xi|\ge\frac{r}{2} }|\langle B_{2}k_{z} , k_{\xi} \rangle| dV(\xi) ~.
		\end{eqnarray*}
		It follows from the third property of weakly localized operators that
		\begin{equation*}
			\lim_{r\to\infty} \sup_{z\in \C^n} \int_{|z-\xi|\ge \frac{r}{2}} I_{z}(\xi) dV(\xi) \le C \lim_{r\to\infty} \sup_{z\in \C^n} \int_{|z-\xi|\ge\frac{r}{2} }|\langle B_{2}k_{z} , k_{\xi} \rangle| dV(\xi) =0~.
		\end{equation*}
		Whence 
		\begin{equation*}
			\lim_{r\to\infty} \sup_{z\in \C^n} \int_{|z-w|\ge r} | \langle B_{1} B_{2} k_{z}, k_{w} \rangle | dV(w) =0 ~.
		\end{equation*}
		We proved the corresponding conditions for \((B_{1}B_{2})^{*}\) in the same way. This finishes the proof.
	\end{proof}
	
	Let \(A\) be a bounded linear operator on a Hilbert space \(H\). We recall (see\cite{zhu2007}) that, if \(A\) is a self-adjoint operator, then 
	\begin{equation*}
		\norm{A^{*}}=\norm{A} = \sup\{|\langle Ax, x \rangle|: \norm{x}=1 \}.
	\end{equation*}
	
	\begin{defn}\label{normstar}
		For an entire function \(h\) in \(\C^{n}\), we write
		\begin{equation*}
			\norm{h}_{*} = \left(\int_{\C^{n}} |h(\zeta)|^{2} e^{-\frac{1}{2}|\zeta|^2} dV(\zeta)\right)^{\frac{1}{2}}~.
		\end{equation*}
		We denote by \(\mathcal{H}_{*}\) the collection of entire functions \(h\) on \(\C^n\) satisfying \(\norm{h}_{*}<\infty\).
	\end{defn}
	\begin{rem}
		The norm \(\norm{\cdot}_{*}\) is equivalent to the norm on the Fock space \(F^{2}_{1/2}\) given in relation  (\ref{Fnorm1}) which is an Hilbert space. More precisely, \(\norm{\cdot}_{*} = (2\pi)^{n/2} \norm{\cdot}_{F^{2}_{1/2}}\). This ensures the continuity of \(\norm{\cdot}_{*}\).
	\end{rem}
	In what follows, we will use the operator \(U_{z}\) defined by 
	\begin{equation*}
		U_{z}f(w) = f(z-w) k_{z}(w)~, \quad f \in H^{2}(\C^n, d\mu)~.
	\end{equation*} 
	
	For any \(f,g \in  H^{2}(\C^n,d\mu)\), let  \(f\otimes g \) be \textbf{the standard tensor product operator } on \(H^{2}(\C^n,d\mu)\) defined by
	\begin{equation}\label{DefOtimes}
		\left(f\otimes g\right)(\cdot)  = \langle \cdot, g \rangle f~.
	\end{equation}
	
	\newpage
	
	\begin{pro}\label{PropPre}
		\phantom{x}	\begin{enumerate}[label=(\alph*)]
			\item \label{itm:first} For \(u \in \Z^{2n}\), \(z\in \C^n\) we have
			\begin{equation}\label{Uu1}
				U_{u}k_{z} = k_{u-z} e^{iIm\langle u, z \rangle}.
			\end{equation}
			Furthermore, we have
			\begin{equation}\label{Uu2}
				U_{u}k_{z} \otimes U_{u}k_{z} = k_{u-z} \otimes k_{u-z}\quad, \quad U_{u}K_{z} \otimes U_{u}K_{z} =e^{|z|^{2}} k_{u-z} \otimes k_{u-z}.
			\end{equation}
			\item \label{itm:second}  For \(f \in L^{\infty}(\C^n, dV)\), we have the following representation for the Toeplitz operator \(T_{f}\)
			\begin{equation}\label{Toef1}
				T_{f} = \frac{1}{\pi^n} \int_{\C^n} f(w) k_{w} \otimes k_{w} ~dV(w).
			\end{equation}
			\item \label{itm:third}  The identity operator \(I_d\) on \(H^{2}(\C^n, d\mu)\) can be expressed as follows:
			\begin{equation*}
				I_{d_{H^2\to H^2}} = \frac{1}{\pi^n} \int_{\C^{n}} k_{z} \otimes k_{z} dV(z) = \int_{S} E_{z} dV(z)~,
			\end{equation*}
			where 
			\begin{equation}\label{Ez1}
				E_{z} = \frac{1}{\pi^n}\sum_{u \in \Z^{2n}} k_{u-z} \otimes k_{u-z}~, \quad  z \in S.
			\end{equation}
			\item  For every \(z\in \C^n\), it holds
			\begin{equation}\label{Limkzkw}
				\lim_{w\to z} \norm{k_{z}-k_{w}}_{*} = 0~.
			\end{equation}
		\end{enumerate}
	\end{pro}
	\begin{proof}
		\begin{enumerate}[label=(\alph*)]
			\item For all \(\xi \in \C^n \), it holds
			\begin{eqnarray*}
				U_{u}k_{z}(\xi) &=& k_{z}(u-\xi) k_{u}(\xi) = e^{\langle u-\xi,z \rangle - \frac{1}{2}|z|^{2}} e^{\langle \xi, u \rangle -\frac{1}{2}|u|^{2}} \\
				&=& e^{-\frac{1}{2}|u-z|^{2}} e^{ \langle \xi,u-z \rangle} e^{i Im\langle u,z \rangle } = k_{u-z}(\xi) e^{iIm\langle u,z \rangle}~.
			\end{eqnarray*}
			The relations in (\ref{Uu2}) follow directly from (\ref{Uu1}) and the fact that \(K_{z}=k_{z} e^{\frac{1}{2}|z|^{2}}\).
			\item For \(h\in H^{2}(\C^n, d\mu) \), we have
			\begin{eqnarray*}
				T_{f}(h)(\xi) &=& \int_{\C^n} f(w) h(w) K(\xi,w) d\mu(w) = \frac{1}{\pi^n} \int_{\C^n} f(w) \langle h, k_{w} \rangle k_{w}(\xi) dV(w) \\
				&=& \frac{1}{\pi^n} \int_{\C^n} f(w)( ( k_{w} \otimes k_{w}) h )(\xi) dV(w).
			\end{eqnarray*}
			\item Let \(f\in H^{2}(\C^n, d\mu) \), by the reproducing property, it holds
			\begin{eqnarray*}
				f(z) &=& \langle f , K_{z} \rangle = \int_{\C^{n}} f(w) K(z,w) d \mu(w) \\
				&=& \frac{1}{\pi^n} \int_{\C^{n}} \langle f, k_{w} \rangle k_{w}(z) dV(w) \\
				&=& \frac{1}{\pi^n} \int_{\C^{n}} ((k_w \otimes k_{w}) f )(z) dV(w)~.
			\end{eqnarray*}
			This combined with the change of variables \(w=u-\xi\) leads to
			\begin{eqnarray*}
				I_{d_{H^2\to H^2}} &=& \frac{1}{\pi^n} \int_{\C^{n}} k_{w} \otimes k_{w} ~ dV(w) = \frac{1}{\pi^n} \sum_{u \in \Z^{2n}} \int_{u-S} k_{w} \otimes k_{w}~ dV(w) \\
				&=& \frac{1}{\pi^n} \sum_{u \in \Z^{2n}} \int_{S} k_{u-\xi} \otimes k_{u-\xi} dV(\xi) = \int_{S} E_{\xi}~ dV(\xi)~.
			\end{eqnarray*}
			\item For \(z\in \C^n\), considering the inner product and the norm in \(F^{2}_{1/2}\), we have 
			\begin{eqnarray*}
				\norm{k_z-k_{w}}^{2}_{*} &=& (2\pi)^{n/2} \left(\norm{k_z}^2 + \norm{k_{w}}^2 - 2Re \langle k_{z}, k_{w} \rangle \right) \\
				&=& (2\pi)^{n/2} \left(e^{\frac{|z|^2}{2}}+e^{\frac{|w|^{2}}{2}} -2Re \langle k_{z}, k_{w} \rangle \right) \underset{w\to z}{\rightarrow }0~.
			\end{eqnarray*}
		\end{enumerate}
	\end{proof}
	
	\begin{rem}
		From point \ref{itm:third} in Proposition \ref{PropPre} , we deduce that for all \(B \in \mathcal{WL}\), 
		\begin{equation}\label{Bint}
			B = \int_{S} \int_{S} E_{w} B E_{z} dV(w) dV(z)~.
		\end{equation}
		Furthermore, for \(z,w \in S\) and \(B\in \mathcal{WL}\),
		\begin{equation}\label{EwBEz}
			E_{w} BE_{z} = \frac{1}{\pi^{2n}} \sum_{u,v \in \Z^{2n}} \langle Bk_{u-z}, k_{v-w} \rangle k_{v-w} \otimes k_{u-z}~.
		\end{equation}
	\end{rem}
	
	\begin{proof}
		In fact, for \(z,w \in S\) and \(f\in L^{\infty}(\C^n, dV)\), we have
		\begin{eqnarray*}
			E_{w}BE_{z} f (\xi) &=& \frac{1}{\pi^n} \sum_{u \in \Z^{2n}} E_{w} B~ (k_{u-z}\otimes k_{u-z})f(\xi) = \frac{1}{\pi^{n}} \sum_{u \in \Z^{2n}} \langle f,k_{u-z} \rangle  E_{w}B~k_{u-z}(\xi)\\
			&=& \frac{1}{\pi^{2n}} \sum_{u,v \in \Z^{2n}} \langle f, k_{u-z} \rangle (k_{v-w}\otimes k_{v-w})Bk_{u-z}(\xi) \\
			& = &  \frac{1}{\pi^{2n}} \sum_{u,v \in \Z^{2n}} \langle f, k_{u-z} \rangle \langle Bk_{u-z}, k_{v-w} \rangle k_{v-w}(\xi) \\
			&=&  \frac{1}{\pi^{2n}} \sum_{u,v \in \Z^{2n}} \langle Bk_{u-z}, k_{v-w} \rangle (k_{v-w} \otimes k_{u-z})f (\xi)~.
		\end{eqnarray*}
		The relation (\ref{Bint}) is obtained by integrating (\ref{EwBEz}) on \(S\times S\), using the fact that \(\C^{n} = \cup_{u \in \Z^{2n}}\{u-S\}\) and the reproducing kernel property. 
	\end{proof}
	
	In what follows, \(\{e_{u}: u \in \Z^{2n}\}\) is any orthonormal basis in \(H^{2}(\C^n,d\mu)\). Let us recall the discrete version of the Schur test, which will be used several times in this paper.
	\begin{lem}\label{SchurTest}
		Let \(K\) be a kernel on \(\N\times \N\). Suppose that \(K(i,j)\ge 0\) for all \(i,j \in \N\) and that there are constants \(C_{1}, C_{2}\) and sequence of strictly positive numbers \(\{h_{j}\}\) such that 
		\begin{equation*}
			\sum_{j=1}^{\infty} K(i,j) h_{j} \le C_{1}h_{i} \quad \text{and} \quad \sum_{j=1}^{\infty}K(j,i)h_{j} \le C_{2}h_{i}
		\end{equation*}
		for every \(i\in \N\). Then for all \(a=(a_{i})\) and \(b=(b_{i})\) in \(l^{2}(\N)\) we have
		\begin{equation*}
			\sum_{j,i=1}^{\infty} K(i,j)|a_{i}||b_{j}| \le (C_{1}C_{2})^{1/2}\norm{a}\norm{b}.
		\end{equation*}
	\end{lem}

	\begin{lem}
		There is a constant \(0< C < \infty\), such that \(\norm{E_{z}} \le C \) for every \(z\in S\).
	\end{lem}
	\begin{proof}
		For \(z\in S\), we define the operator 
		\begin{equation*}
			F_{z} = \frac{1}{\pi^{n/2}} \sum_{u \in \Z^{2n}} e_{u} \otimes k_{u-z}.
		\end{equation*}
		Then we have 
		\begin{equation*}
			F^{*}_{z} = \frac{1}{\pi^{n/2}}\sum_{u \in \Z^{2n}} k_{u-z}\otimes e_{u} \text{ and }E_{z} = F^{*}_{z}F_{z}~.
		\end{equation*}
		In fact, for \(g\in H^{2}(\C^n,d\mu)\), 
		\begin{eqnarray*}
			F^{*}_{z}F_{z} g(\xi) &=& \frac{1}{\pi^{n/2}} \sum_{u \in \Z^{2n}} \langle g, k_{u-z} \rangle F^{*}_{z} e_{u} (\xi)~\\
			&=& \frac{1}{\pi^n} \sum_{u,v \in \Z^{2n}} \langle g, k_{u-z} \rangle \langle e_{u},e_{v} \rangle k_{v-z} (\xi) \\
			&=& \frac{1}{\pi^n} \sum_{u \in \Z^{2n}} \langle g, k_{u-z} \rangle k_{u-z}(\xi) = E_{z}g(\xi)~.
		\end{eqnarray*}
		Moreover, 
		\begin{equation*}
			F_{z}F^{*}_{z} = \frac{1}{\pi^n} \sum_{u,v \in \Z^{2n}} \langle k_{u-z}, k_{v-z} \rangle e_{v} \otimes e_{u}~.
		\end{equation*}
		Since \(  \norm{F^{*}_{z}F_{z}}=\norm{F_{z}F^{*}_{z}}\), to get \( \norm{E_{z}}\) it suffices to estimate the latter. For every vector \(x=\sum_{u \in \Z^{2n}} x_{u}e_{u}\) of \(H^{2}(\C^n,d\mu)\), we have
		\begin{eqnarray}\label{FFx}
			\langle F_{z}F^{*}_{z} x, x \rangle &=& \frac{1}{\pi^{n}} \sum_{u,v \in \Z^{2n}} \langle k_{u-z}, k_{v-z} \rangle \langle x, e_{u} \rangle \langle e_{v},x \rangle \nonumber \\
			&\le &  \frac{1}{\pi^{n}} \sum_{u,v \in \Z^{2n}} |\langle k_{u-z}, k_{v-z} \rangle| |x_{u}| |x_{v}| \nonumber \\
			&=&  \frac{1}{\pi^{n}} \sum_{u,v \in \Z^{2n}}  e^{-\frac{1}{2}|u-v|^{2}}  |x_{u}| |x_{v}|~.
		\end{eqnarray}
		Since \(\sum_{u \in \Z^{2n}}  e^{-\frac{1}{2}|u-v|^{2}} = \sum_{u \in \Z^{2n}} e^{-\frac{1}{2}|u|^2} \) for \(v\in \Z^{2n}\), then the function \(A(u,v) = e^{-\frac{1}{2}|u-v|^{2}}\) satisfies the hypotheses of the discrete Schur test with \(h_{u}=1\) for all \(u \in \Z^{2n}\). Hence, from (\ref{FFx}) we have
		\begin{equation*}
			\langle F_{z}F^{*}_{z} x, x \rangle \le \frac{1}{\pi^n}C \norm{x}^2~.
		\end{equation*}
		Since \( F_{z}F^{*}_{z}\) is self-adjoint, it follows that \(\norm{E_{z}} =\norm{ F_{z}F^{*}_{z}} \le C\).
	\end{proof}

	\begin{lem}\label{LemUh}
		There is a constant \(0<C<\infty\) such that the following estimate holds: Let \(h_{u} \in \mathcal{H}_{*}, u \in \Z^{2n}\), be functions satisfying the condition \(\sup_{u\in \Z^{2n}}\norm{h_{u}}_{*}<\infty\). Then
		\begin{equation*}
			\norm{\sum_{u \in \Z^{2n}}(U_{u}h_{u})\otimes e_{u}} \le C \sup_{u\in \Z^{2n}} \norm{h_{u}}_{*}~.
		\end{equation*}
	\end{lem}
	
	\begin{proof}
		We start by estimating \(|\langle U_{u}h_{u},U_{v}h_{v}\rangle|\). For \(u,v \in \Z^{2n} \), we have
		\begin{eqnarray}\label{UuUv1}
			\langle U_{u} h_{u} , U_{v} h_{v} \rangle & = & \frac{1}{\pi^n} \int_{\C^{n}} U_{u} h_{u} (\xi) ~\overline{U_{v} h_{v}}(\xi)~ e^{-|\xi|^{2}} dV(\xi) \nonumber \\
			&=& \frac{1}{\pi^n} \int_{\C^{n}} h_{u}(u-\xi) ~\overline{h_{v}(v-\xi)}~ k_{u}(\xi) \overline{k_{v}(\xi)} ~ e^{-|\xi|^{2}} dV(\xi)~.
		\end{eqnarray}
		From relation (\ref{kzw}), we have
		\begin{equation*}
			|k_{u}(\xi) \overline{k_{v}(\xi)}|  e^{-|\xi|^{2}} = e^{-\frac{1}{2}(|u-\xi|^2 + |v-\xi|^2)} \le e^{-\frac{1}{8}|u-v|^2} e^{-\frac{1}{4}|u-\xi|^2} e^{-\frac{1}{4}|v-\xi|^2}~.
		\end{equation*}
		Combining this with relation (\ref{UuUv1}) and applying \(\mathrm{H\ddot{o}lder}\) inequality, we obtain
		\begin{equation}\label{UuUv2}
			|\langle U_{u} h_{u} , U_{v} h_{v} \rangle| \le \frac{1}{\pi^n} e^{-\frac{1}{8}|u-v|^2} \norm{h_{u}}_{*} \norm{h_{v}}_{*} \le  \frac{1}{\pi^n} e^{-\frac{1}{8}|u-v|^2} H^{2}_{*}~,
		\end{equation}
		where \(H_{*}=\sup_{u\in \Z^{2n}}\norm{h_{u}}_{*}\).
		
		We consider the operator \(A\) defined by
		\begin{equation*}
			A = \sum_{u \in \Z^{2n}} (U_{u} h_{u}) \otimes e_{u}~.
		\end{equation*}
		For any vector \(x=\sum_{u \in \Z^{2n}} x_{u} e_{u} \in H^{2}(\C^n, d\mu)\), using (\ref{UuUv2}) we have
		\begin{eqnarray*}
			\norm{Ax}^{2} &= & \langle Ax, Ax \rangle = \sum_{u,v \in \Z^{2n}} \langle (U_{u}h_{u})\otimes e_{u} x, (U_{v}h_{v})\otimes e_{v}x \rangle \\
			&=& \sum_{u,v \in \Z^{2n}} \langle x, e_{u} \rangle \langle e_{v}, x \rangle \langle U_{u} h_{u} , U_{v} h_{v} \rangle = \sum_{u,v \in \Z^{2n}}  \overline{x_{u}} x_{v} \langle U_{u} h_{u} , U_{v} h_{v} \rangle \\
			&\le &  \sum_{u,v \in \Z^{2n}}  |\overline{x_{u}}| |x_{v}| |\langle U_{u} h_{u} , U_{v} h_{v} \rangle|  \le \frac{1}{\pi^n} H^{2}_{*} \sum_{u,v \in \Z^{2n}} e^{-\frac{1}{8}|u-v|^2} |x_{u}| |x_{v}|~.
		\end{eqnarray*}
		The discrete Schur test (see Lemma \ref{SchurTest}) applied to the right-hand side (with \(h_{u}=1 \,\forall u \in \Z^{2n}\)) of the later inequality, leads to
		\begin{equation*}
			\norm{Ax}^{2} \le C H^{2}_{*} \sum_{u \in \Z^{2n}} |x_{u}|^{2} = C H^{2}_{*} \norm{x}^{2},
		\end{equation*}
		where \(C=\sum_{u\in \Z^{2n}} e^{-\frac{1}{8}|u|^{2} }\) is finite. Since the vector \(x\) is arbitrary, we conclude that \mbox{\(\norm{A}\le C^{\frac{1}{2}} H_{*}\)}.
	\end{proof}
	
	\begin{pro}\label{proYz}
		Suppose \(\{c_{u}: u \in \Z^{2n}\}\) are complex numbers satisfying the condition \mbox{\(\sup_{u\in \Z^{2n}}|c_{u}| < \infty\)}. Then for each \(z\in \C^{n}\), the operator
		\begin{equation}\label{Yz}
			Y_{z} = \sum_{u \in \Z^{2n}} c_{u} k_{u-z} \otimes k_{u-z}
		\end{equation}
		is bounded on \(H^{2}(\C^n, d\mu)\). Moreover, the map \(z \mapsto Y_{z}\) from \(\C^{n}\) to \(\mathcal{B}(H^{2}(\C^{n},d\mu))\) is continuous with respect to the operator norm.
	\end{pro}
	\begin{proof}
		Let \(\{c_{u} : u \in \Z^{2n}\}\) be complex numbers such that \mbox{\(\sup_{u\in \Z^{2n}}|c_{u}| < \infty\)}. For \(z\in \C^{n}\), we define 
		\begin{equation*}
			A_{z} = \sum_{u \in \Z^{2n}} c_{u} (U_{u}k_{z}) \otimes e_{u} \quad \text{  and  } \quad B_{z} = \sum_{u \in \Z^{2n}}(U_{u}k_{z})\otimes e_{u}~.
		\end{equation*}
		Using the relation (\ref{Uu2}), we have \(Y_{z} = A_{z}B^{*}_{z}\). Set \(h_{u} = c_{u}k_{z}\), then using Definition \ref{normstar}, we have
		\begin{equation*}
			\norm{h_{u}}_{*} = |c_{u}|\norm{k_{z}}_{*} = (2\pi)^{n/2} e^{\frac{|z|^2}{2}} |c_{u}|~.
		\end{equation*}
		Therefore, \(\sup_{u\in \Z^{2n}} \norm{h_{u}}_{*} < \infty\). Applying Lemma \ref{LemUh} to \(h_{u} = c_{u}k_{z}\), we see that \(\norm{A_{z}} \le C \sup_{u\in \Z^{2n}}|c_{u}|\).  Thus each \(A_{z}\) is bounded. Since \(B_z\) is just a special case of \(A_{z}\) (with \(c_{u}=1\) for all \(u \in \Z^{2n}\)), then it is also bounded and hence \(Y_{z} = A_{z}B^{*}_{z}\) is bounded. 
		
		To show that the map \(z\mapsto Y_{z}\) is continuous with respect to the operator norm, it suffices to show that the maps \(z\mapsto A_{z}\) and \(z\mapsto B_{z}\) are continuous.
		
		For any \(z,w \in \C^n\), we have
		\begin{equation*}
			A_{z} - A_{w} = \sum_{u \in \Z^{2n}} c_{u}\left\{U_{u}(k_{z}-k_{w})\right\}\otimes e_{u}~.
		\end{equation*}
		Applying Lemma \ref{LemUh} in the case \(h_{u}=c_{u}(k_{z}-k_{w})\) and using (\ref{Limkzkw}) , it holds
		\begin{equation*}
			\norm{A_{z}-A_{w}} \le C \left(\sup_{u\in \Z^{2n}} |c_{u}|\right) \norm{k_{z}-k_{w}}_{*} \underset{w\to z}{\rightarrow}0~.
		\end{equation*}
		Hence the map \(z \mapsto A_{z}\) is continuous with respect to the operator norm. Similarly, we show that the map \(z\mapsto B_{z}\) is also continuous. This completes the proof.
	\end{proof}
	
	\begin{pro}\label{BT1}
		For all \(B\in \mathcal{WL}\) and \(z,w\in S\), we have \(E_{w}BE_{z} \in \mathcal{T}^{(1)}\).
	\end{pro}
	To prove this proposition, we introduce the following definitions that will allow us to split his proof into two independent parts.
	
	
	\begin{defn}\label{DoD}
		\phantom{x}	\begin{enumerate}[label=(\alph*)]
			\item We denote by \(\mathcal{D}_{0}\) the collection of operators of the form
			\begin{equation*}
				\sum_{u \in \Z^{2n}} c_{u}k_{u} \otimes k_{\gamma(u)}~,
			\end{equation*}
			where \(\{c_{u} : u \in \Z^{2n}\}\) is any bounded set of complex coefficients and \(\gamma : \Z^{2n} \rightarrow \C^{n}\) is any map for which there exists \(0<C < \infty\) such that \(\norm{u-\gamma(u)} \le C\) for every \(u \in \Z^{2n}\).
			\item Let \(\mathcal{D}\) denote the operator-norm closure of the linear span of \(\mathcal{D}_{0}\).
		\end{enumerate}
	\end{defn}
	
	\begin{pro}\label{EwBEzD}
		If \(B \in \mathcal{WL}\), then \(E_{w}BE_{z} \in \mathcal{D} \) for all \(z,w \in \C^n\). 
	\end{pro}
	
	The following lemma is necessary for the proof of Proposition \ref{EwBEzD}.
	
	\begin{lem}\label{Lim=0}
		Let \(B \in \mathcal{WL} \), then for every \(z,w \in \C^n \), we have
		\begin{equation*}
			\lim_{R\to \infty} \sup_{u\in \Z^{2n}} \sum_{\underset{|u-v|>R}{v\in \Z^{2n} }} |\langle B k_{u-z}, k_{v-w} \rangle|  =0 \quad \text{ and } \quad \lim_{R\to \infty} \sup_{u\in \Z^{2n}} \sum_{\underset{|u-v|>R}{v\in \Z^{2n} }} |\langle k_{u-z}, Bk_{v-w}\rangle| = 0~.
		\end{equation*}
	\end{lem}
	\begin{proof}
		By \cite[Lemma 2.32]{zhu2012}, for any entire function \(f\) on \(\C^n\), we have 
		\begin{equation*}
			\left|f(z)e^{-\frac{\alpha}{2}|z|^2}\right|^p \le C \int_{B(z,\delta)} |f(w)e^{-\frac{\alpha}{2}|w|^2}|^p dV(w) \quad \text{ for } z \in \C^n.
		\end{equation*}
		Hence for \(\alpha=p=1\) and \(\delta\) small such that the balls \(\{B(v-w,\delta): v \in \Z^{2n} \}\) are mutually disjoint, we have
		\begin{equation*}
			|\langle B k_{u-z}, k_{v-w} \rangle| = |Bk_{u-z} (v-w) |e^{-\frac{|v-w|^2}{2}} \le C \int_{B(v-w,\delta)} |Bk_{u-z}(\zeta) | e^{-\frac{|\zeta|^2}{2}} dV(\zeta).
		\end{equation*}
		Indeed, for \(\delta < \frac{1}{2} \), the balls \(\{B(v-w,\delta) : v \in \Z^{2n} \}\) are mutually disjoint. Otherwise, there would exist \(v,v' \in \Z^{2n} \) such that \(v \neq v' \), and a point \(\xi\) such that \(\xi \in B(v-w,\delta) \cap B(v'-w,\delta) \). In other words:
		\begin{equation*}
			|v-w-\xi| < \delta \quad \text{ and } \quad |v'-w - \xi| < \delta~.
		\end{equation*}
		This implies that 
		\begin{equation*}
			|v-v'|=|(v-w-\xi)-(v'-w-\xi)|\le |v-w-\xi| + |v'-w-\xi| < \delta + \delta= 2 \delta < 1~.
		\end{equation*}
		That is \(|v-v'|<1\). This contradicts the well-known fact that \(|v-v'|\ge 1\). This result actually implies that there exists \(N\in \N\) such that each \(\zeta\in \C^n \) belongs to at most \(N\) balls in \(\{B(v-w,\delta): v\in \Z^{2n}\}\). That is \(\sum_{v \in \Z^{2n}} \chi_{B(v-w,\delta)}(\zeta) \le N \) for each \(\zeta \in \C^n \).
		
		For \(\zeta \in B(v-w,\delta) \),  we have \(|v-w-\zeta|< \delta\). Hence, considering any \(R> \delta + |z-w|\), we have
		\begin{equation*}
			|u-z-\zeta| = |u-v + v-w-\zeta + w-z| \ge |u-v|-|v-w-\zeta|-|w-z| > R - \delta -|z-w|~,
		\end{equation*}
		and 
		\begin{eqnarray*}
			\sum_{\underset{|u-v|>R}{v\in \Z^{2n} }} |\langle B k_{u-z}, k_{v-w} \rangle| &\le& C \int_{|u-z-\zeta|> R-\delta-|z-w|} \sum_{v \in \Z^{2n}} \chi_{B(v-w,\delta)}(\zeta) |\langle Bk_{u-z}, k_{\zeta} \rangle| dV(\zeta) \\
			&\le& CN \int_{|u-z-\zeta|> R-\delta-|z-w|}|\langle Bk_{u-z}, k_{\zeta} \rangle| dV(\zeta),
		\end{eqnarray*}
		which tends to \(0\) as \(R\to \infty\) from the third condition of weakly localized operators. 
	\end{proof}

		\begin{proof}[\textbf{Proof of Proposition \ref{EwBEzD} }]
			From (\ref{EwBEz}), we have 
			\begin{equation*}
				E_{w} BE_{z} = \frac{1}{\pi^{2n}} \sum_{u,v \in \Z^{2n}} \langle Bk_{u-z}, k_{v-w} \rangle k_{v-w} \otimes k_{u-z}~.
			\end{equation*}
			Thus for any \(R>0\), we can write \(E_{w}BE_{z} = V_{R} + W_{R}\), where
			\begin{equation*}
				V_{R} = \frac{1}{\pi^{2n}} \sum_{\underset{|u-v|\le R}{u,v \in \Z^{2n}}} \langle Bk_{u-z}, k_{v-w} \rangle k_{v-w} \otimes k_{u-z} \quad \text{ and } 
			\end{equation*}
			\begin{equation*}
				\quad W_{R} = \frac{1}{\pi^{2n}} \sum_{\underset{|u-v|>R}{u,v \in \Z^{2n}}} \langle Bk_{u-z}, k_{v-w} \rangle k_{v-w} \otimes k_{u-z}~.
			\end{equation*}
			To complete the proof, it suffices to prove that:
			\begin{enumerate}[label=(\alph*)]
				\item \label{itm:first1} \(\lim_{R\to \infty} \norm{W_{R}}=0\).
				\item \label{itm:second1} \(V_{R} \in \) span(\(\mathcal{D}_{0}\)) for every \(R>0\).
			\end{enumerate}
			
			\textbf{Let us prove} \ref{itm:first1}. For every \(h \in H^{2}(\C^n, d\mu) \), using (\ref{Uu1}) we have 
			\begin{eqnarray*}
				\norm{\sum_{u \in \Z^{2n}} e_{u} \otimes U_{u}k_{z} ~h }^{2} &=& \sum_{u,v \in \Z^{2n}} \langle  h,U_{u}k_{z} \rangle \langle U_{v}k_{z} , h \rangle \langle e_{u}, e_{v} \rangle \\ 
				&=& \sum_{u \in \Z^{2n}} |\langle h, U_{u}k_{z}\rangle |^2 = \sum_{u \in \Z^{2n}} |\langle h,k_{u-z} \rangle|^{2}~.
			\end{eqnarray*}
			From Lemma \ref{LemUh}, there are constants \(C_{1}, C_{2}\), such that
			\begin{equation}\label{hkw}
				\sum_{u \in \Z^{2n}} |\langle h,k_{u-z} \rangle|^{2} \le C_{1} \norm{h}^{2} \quad \text{ and } \quad \sum_{v \in \Z^{2n}} |\langle h,k_{v-w} \rangle|^{2} \le C_{2} \norm{h}^{2}~.
			\end{equation}
			Given \(h,g \in H^{2}(\C^n,d\mu)\), we have 
			\begin{eqnarray*}
				|\langle W_{R}h,g \rangle| & \le & \frac{1}{\pi^{2n}} \sum_{\underset{|u-v|>R}{u,v \in \Z^{2n}}} | \langle Bk_{u-z}, k_{v-w} \rangle || \langle h, k_{u-z} \rangle| | \langle k_{v-w} , g \rangle |~.
			\end{eqnarray*}
			Applying the Schur test to this inequality and combining with (\ref{hkw}), we obtain
			\begin{eqnarray*}
				|\langle W_{R}h,g \rangle| &\le & \{H(R)G(R)\}^{\frac{1}{2}} \left(\sum_{u \in \Z^{2n}} |\langle h, k_{u-z} \rangle|^2 \right)^{\frac{1}{2}} \left(\sum_{v \in \Z^{2n}} |\langle h, k_{v-w} \rangle|^2 \right)^{\frac{1}{2}} \\
				&\le & \{C_{1}C_{2} H(R)G(R)\}^{\frac{1}{2}} \norm{h}\norm{g}~,
			\end{eqnarray*}
			where
			\begin{equation*}
				H(R) = \sup_{u\in \Z^{2n}} \sum_{\underset{|u-v|>R}{v\in \Z^{2n} }} |\langle Bk_{u-z}, k_{v-w} \rangle| \quad \text{ and } \quad G(R) = \sup_{v\in \Z^{2n}} \sum_{\underset{|u-v|>R}{u\in \Z^{2n} }} |\langle Bk_{u-z}, k_{v-w} \rangle|~.
			\end{equation*}
			Since \(h,g \in H^{2}(\C^n,d\mu) \) are arbitrary, this leads to
			\begin{equation*}
				\norm{W_{R}} \le \{C_{1}C_{2} H(R)G(R)\}^{\frac{1}{2}}~.
			\end{equation*}
			From Lemma \ref{Lim=0}, we have \(\lim_{R\to \infty} H(R) = 0\) and \(\lim_{R\to \infty} G(R) = 0\). Therefore \(\lim_{R\to \infty} W_{R}=0\).
			
			\textbf{Let us prove} \ref{itm:second1}. That is \(V_{R} \in \) span(\(\mathcal{D}_{0}\)) for every \(R>0\). 
			
			For \(R>0\) and \(v\in \Z^{2n}\), we define \(F_{v} = \{u \in \Z^{2n} : |u-v|\le R \}\). Since \(\Z^{2n}\) is a lattice, there is an \(N \in \N\) such that Card\((F_{v})\le N \) for every \(v \in \Z^{2n}\). Also, we recall that if \(v,v^{'}\in \Z^{2n}\) and \(v \neq v^{'}\) then \(|v-v^{'}|\ge 1\). Then, we write \(V_{R}\) as follows
			\begin{eqnarray*}
				V_{R}= \frac{1}{\pi^{2n}} \sum_{v \in \Z^{2n}} \sum_{u\in F_{v}}\langle Bk_{u-z}, k_{v-w} \rangle k_{v-w} \otimes k_{u-z} .
			\end{eqnarray*}
			To prove \ref{itm:second1}, we define for each \(j \in \{1,\dots, N\}\), the following sets:
			\begin{equation*}
				\Gamma_{j} = \{v \in \Z^{2n}: \text{Card}(F_{v}) = j \} \quad \text{ and } \quad K_{j} = \{v-w: v \in \Gamma_{j}\}~.
			\end{equation*}
			Then \(V_{R} = \frac{1}{\pi^{2n}} \left(X_{1} + \cdots + X_{N}\right)\), where 
			\begin{equation*}
				X_{j} = \sum_{v\in \Gamma_{j}} \sum_{u\in F_{v}}\langle Bk_{u-z}, k_{v-w} \rangle k_{v-w} \otimes k_{u-z} = \sum_{v-w\in K_{j}} \sum_{u\in F_{v}}\langle Bk_{u-z}, k_{v-w} \rangle k_{v-w} \otimes k_{u-z} ~.
			\end{equation*}
			Thus what remains is to show that \(X_{j}  \in \) span(\(\mathcal{D}_{0}\)) for every \(j\). For all \(j\) we can define maps 
			\begin{equation*}
				\gamma^{1}_{j}, \cdots, \gamma^{j}_{j} : K_{j} \rightarrow \C^{n}
			\end{equation*}
			such that \(\{u-z: u \in F_{v}\} = \{\gamma^{1}_{j}(v-w), \cdots, \gamma^{j}_{j} (v-w) \}\) for every \(v \in \Gamma_{j} \). Thus \( X_{j}=  X^{1}_{1} + \cdots + X^{j}_{N}  \), where for each \(\nu \in \{1,\dots,j\}\) we have
			\begin{equation*}
				X^{\nu}_{j} = \sum_{\xi \in K_{j}} \langle Bk_{\gamma^{\nu}_{j}(\xi)}, k_{\xi} \rangle k_{\xi} \otimes k_{\gamma^{\nu}_{j}(\xi)}.
			\end{equation*}
			Referring to the above definitions, for each \(j, \nu\), if \(\xi \in K_{j}\) there exist \(v\in \Gamma_{j}\) and \(u\in F_{v}\) such that \(\xi = v-w\) and \( \gamma^{\nu}_{j} (\xi) = u-z \). Therefore
			\begin{equation*}
				|\xi-\gamma^{\nu}_{j}(\xi)| = |v-w -u + z| \le R + |w| + |z|.
			\end{equation*}
			We deduce from Definition \ref{DoD}  of \(\mathcal{D}_{0}\), that \(X^{\nu}_{j} \in \mathcal{D}_{0} \). This ends the proof.
		\end{proof}

		\begin{pro}\label{DoT1}
			We have \(\mathcal{D}_{0} \subset \mathcal{T}^{(1)}\).
		\end{pro}
		To establish the proof of this proposition, we will need the next three propositions.
		\begin{pro}\label{YzT1}
			Suppose that \(\{c_{u}: u \in \Z^{2n}\}\) is a bounded set of complex coefficients. Then for each \(z \in \C^{n}\), the operator \(Y_{z}\) defined in (\ref{Yz}) belongs to \(\mathcal{T}^{(1)}\). 
		\end{pro}
		\begin{proof}
			\begin{enumerate}[label=(\alph*)]
				\item Let us first show that \(Y_{0} \in \mathcal{T}^{(1)}\). We have \(|u-v|\ge 1\) for all \(u \neq v \in \Z^{2n}\). Hence \(B(u,\frac{1}{2}) \cap B(v,\frac{1}{2}) = \emptyset \) for \(u\neq v\). For each \(0 < \varepsilon < \frac{1}{2}\), define the operator
				\begin{equation*}
					A_{\varepsilon} = \frac{1}{|B(0,\varepsilon)|} \int_{B(0,\varepsilon)} Y_{z} dV(z) .
				\end{equation*}
				From Proposition \ref{proYz}, we have the norm continuity of the map \(z\mapsto Y_{z}\) and it implies that
				\begin{equation*}
					\lim_{\varepsilon \to 0} \norm{Y_{0}-A_{\varepsilon}} = 0.
				\end{equation*}
				This comes from the fact that
				\begin{eqnarray*}
					\norm{Y_{0}-A_{\varepsilon}} &=& \norm{\frac{1}{|B(0,\varepsilon)|} \int_{B(0,\varepsilon)} (Y_{0}-Y_{z}) ~dV(z)} \le \frac{1}{|B(0,\varepsilon)|} \int_{B(0,\varepsilon)} \norm{Y_{0}-Y_{z}} ~dV(z)
				\end{eqnarray*}
				and \(\lim_{z\to 0}\norm{Y_{z}-Y_{0}}=0\) .
				
				Thus to prove the membership \(Y_{0}\in \mathcal{T}^{(1)}\), it suffices to show that each \(A_{\varepsilon}\) is a Toeplitz operator with bounded symbol. Indeed, with the change of variables \(w=u-z\), we have
				\begin{eqnarray*}
					A_{\varepsilon} &=& \frac{1}{|B(0,\varepsilon)|} \int_{B(0,\varepsilon)} Y_{z} dV(z) = \frac{1}{|B(0,\varepsilon)|} \sum_{u \in \Z^{2n}}  \int_{B(0,\varepsilon)} c_{u} k_{u-z} \otimes k_{u-z}~ dV(z) \\
					&=& \frac{1}{|B(0,\varepsilon)|} \sum_{u \in \Z^{2n}}  \int_{B(u,\varepsilon)} c_{u} k_{w} \otimes k_{w} ~dV(w) = \frac{1}{\pi^n} \int_{\C^n} f_{\varepsilon}(w) k_{w} \otimes k_{w} ~dV(w)~,
				\end{eqnarray*}
				where 
				\begin{eqnarray*}
					f_{\varepsilon}(w) = \frac{\pi^n}{|B(0,\varepsilon)|} \sum_{u \in \Z^{2n}}c_{u} \chi_{B(u,\varepsilon)}(w)
				\end{eqnarray*}
				belongs to \(L^{\infty}(\C^n,dV)\), since \(0<\varepsilon<\frac{1}{2}\) and \(B(u,\varepsilon)\cap B(v,\varepsilon) = \emptyset\) for \(u\neq v \in \Z^{2n} \). From (\ref{Toef1}) we observe that \(A_{\varepsilon} = T_{f_{\varepsilon}}\). Whence \(Y_{0} \in \mathcal{T}^{(1)} \).
				\item Let \(z \in \C^{n}\). There is a partition \(\Z^{2n} = \Gamma_{1} \cup \cdots \cup\Gamma_{m} \) such that for every \(i \in \{1,\dots,m\}\), \(|u-v|\ge 1\) for \(u\neq v \in \Gamma_{i}\). Set \(K_{i} = \{u-z: u \in \Gamma_{i}\}\). We have \(Y_{z} = Y_{z,1} + \cdots + Y_{z,m}\), where 
				\begin{equation*}
					Y_{z,i} = \sum_{u-z \in K_{i}} c_{u}k_{u-z} \otimes k_{u-z},
				\end{equation*}
				for all \(i \in \{1,\dots,m\}\). By (a), we have \(Y_{z,i} \in \mathcal{T}^{(1)}\) for all \(i \in \{1,\dots,m\}\). Hence \(Y_{z} \in \mathcal{T}^{(1)}\).
			\end{enumerate}
		\end{proof}
		
		To continue our work, we need to introduce the following functions. For each pair \(\alpha \in \N^{n} \) and \(z \in \C^{n} \), we define
		\begin{equation*}
			K_{z;\alpha}(\zeta) = \zeta^{\alpha} e^{\langle \zeta,z \rangle} = \zeta^{\alpha}K_{z}(\zeta)~, \qquad \zeta \in \C^n,
		\end{equation*}
		where \(\alpha = (\alpha_1, \dots, \alpha_n) \). We recall that \(|\alpha| = \alpha_1 + \cdots + \alpha_n\) and \(\zeta^{\alpha} = \zeta^{\alpha_1}_1 \cdots \zeta^{\alpha_n} \).
		
		\begin{pro}\label{PropUuKzT1}
			Let \(\{c_{u} : u \in \Z^{2n}\}\) be a bounded set of complex numbers. For every pair \(\alpha \in \N^{n} \) and \(z \in \C^{n} \), we have
			\begin{equation*}
				\sum_{u \in \Z^{2n}} c_{u}(U_{u}K_{z}) \otimes (U_{u}K_{z;\alpha}) \in \mathcal{T}^{(1)}.
			\end{equation*}
		\end{pro}
		
		\begin{proof}
			We will prove this proposition by an induction on \(|\alpha|\).
			
			If \(|\alpha|=0\), that is \(\alpha=0\), then from (\ref{Uu2}) and Proposition \ref{YzT1} , it holds
			\begin{eqnarray*}
				\sum_{u \in \Z^{2n}} c_{u}(U_{u}K_{z}) \otimes (U_{u}K_{z;0} ) &=& \sum_{u \in \Z^{2n}} c_{u}(U_{u}K_{z} ) \otimes (U_{u}K_{z}) \\
				& =& e^{|z|^{2}}\sum_{u \in \Z^{2n}} c_{u}~ k_{u-z} \otimes k_{u-z}=  e^{|z|^2} Y_{z}\in \mathcal{T}^{(1)}~.
			\end{eqnarray*}
			
			Let \(k \in \N \), assume that the proposition is true for every \(\alpha \in \N^n \) satisfying the condition \(|\alpha|\le k\). Now consider the case where \(\alpha \in \N^n \) is such that \(|\alpha|=k+1\). Then, we can decompose \(\alpha\) in the form \( \alpha=a+b \), where \(|a|=k\) and \(|b|=1\). Thus, there exists \(\nu \in \{1,\dots,n\} \) such that \(b_{\nu}=1\) and \(b_{j}=0\) for \(j\neq \nu\). By the induction hypothesis , we have 
			\begin{equation}\label{Uza}
				\sum_{u \in \Z^{2n}} c_{u} (U_{u}K_{z}) \otimes (U_{u}K_{z;a}) \in \mathcal{T}^{(1)} \qquad \text{ for every } z \in \C^{n}~.
			\end{equation}
			Let \(z\in \C^n \). For each \(t> 0 \), we define the following operators
			\begin{equation*}
				A_{t} = \sum_{u \in \Z^{2n}} c_{u} (U_{u} K_{z+tb}) \otimes (U_{u}K_{z+tb;a}) \quad \text{ and } \quad B_{t} = \sum_{u \in \Z^{2n}} c_{u} (U_{u} K_{z+itb}) \otimes (U_{u}K_{z+itb;a})~.
			\end{equation*}
			We also define 
			\begin{equation*}
				X = \sum_{u \in \Z^{2n}} c_{u} \{ (U_{u} K_{z}) \otimes (U_{u}K_{z;\alpha}) + (U_{u}K_{z;b}) \otimes (U_{u} K_{z;a}) \} \text{ and } 
			\end{equation*}
			\begin{equation*}
				Y = \sum_{u \in \Z^{2n}}c_{u} \{ (U_{u}K_{z}) \otimes (U_{u}K_{z;\alpha}) - (U_{u}K_{z;b} \otimes (U_{u}K_{z;a}) ) \}.
			\end{equation*}
			We will show that 
			\begin{align}
				\lim_{t \to 0} \norm{ \frac{1}{t}(A_{t}-A_{0}) -X}  &= 0\,,  \label{AtA0} \quad \text{ and } \\
				\lim_{t \to 0} \norm{ \frac{1}{it} (B_{t}-B_{0}) -Y } &= 0. \label{BtB0}
			\end{align}
			Before getting to their proofs, let us first see the consequence of these limits. By (\ref{Uza}), we have \( A_{t} \in \mathcal{T}^{(1)} \)  and \( B_{t} \in \mathcal{T}^{(1)} \) for all \(t>0\). Hence  (\ref{AtA0}) and  (\ref{BtB0}) implies that \(X,Y \in \mathcal{T}^{(1)} \). Therefore
			\begin{equation*}
				\sum_{u \in \Z^{2n}} c_{u}(U_{u}K_{z}) \otimes (U_{u}K_{z;\alpha}) = \frac{1}{2} (X+Y) \in \mathcal{T}^{(1)},
			\end{equation*}
			completing the induction on \(|\alpha|\).
			
			Let us prove (\ref{AtA0}). We have \(\frac{1}{t}(A_{t}-A_{0}) = G_{t} + H_{t} \), where
			\begin{align*}
				H_{t} &= \frac{1}{t} \sum_{u \in \Z^{2n}} c_{u} (U_{u}K_{z+tb}) \otimes \{U_{u}(K_{z+tb;a} - K_{z;a}) \} \quad \text{ and }\\
				G_{t} &= \frac{1}{t} \sum_{u \in \Z^{2n}} c_{u} \{ U_{u}(K_{z+tb} - K_{z}) \} \otimes (U_{u}K_{z;a})~.
			\end{align*}
			Similarly, we write \(X= V+ W\), where
			\begin{equation*}
				V = \sum_{u \in \Z^{2n}} (U_{u}K_{z}) \otimes (U_{u}K_{z;\alpha}) \quad \text{ and } \quad W = \sum_{u \in \Z^{2n}} c_{u} (U_{u}K_{z;b}) \otimes (U_{u}K_{z;a})~.
			\end{equation*}
			Then \( \norm{ \frac{1}{t}(A_{t}-A_{0}) -X} \le \norm{H_{t} -V} + \norm{ G_{t} -W} \), and (\ref{AtA0}) will follow if we prove that 
			\begin{align}
				\lim_{t\to 0} \norm{H_{t} - V} &= 0 \quad \text{ and } \label{HtV} \\
				\lim_{ t \to 0} \norm{ G_{t} -W} &=0~. \label{GtW}
			\end{align}
			To prove (\ref{HtV}), we write \(H_{t} -V = S_{t} + T_{t} \), where
			\begin{align*}
				S_{t} &= \sum_{u \in \Z^{2n}} c_{u} (U_{u}K_{z+tb}) \otimes \{ U_{u}(t^{-1} (K_{z+tb;a} - K_{z;a}) -K_{z;\alpha} ) \} \quad  \text{  and  } \\
				T_{t} &= \sum_{u \in \Z^{2n}} c_{u}\{ U_{u} (K_{z+tb} - K_{z}) \} \otimes (U_{u}K_{z;\alpha}).
			\end{align*}
			Thus to prove (\ref{HtV}), we just have to prove that \(\lim_{t \to 0} \norm{S_{t}} = 0 \) and \(\lim_{t \to 0}\norm{T_{t}} = 0 \). To do this, we factor \(S_{t}\) in the form \(S_{t} = S^{(1)}_{t} \left(S^{(2)}_{t}\right)^{*} \) where 
			\begin{eqnarray*}
				S^{(1)}_{t} = \sum_{u \in \Z^{2n}} c_{u} (U_{u}K_{z+tb}) \otimes e_{u} \quad \text{ and } \quad  S^{(2)}_{t} = \sum_{u \in \Z^{2n}} \{ U_{u}(t^{-1} ( K_{z+tb;a}-K_{z;a}) -K_{z;\alpha}) \}\otimes e_{u}~.
			\end{eqnarray*}
			Then it follows from Lemma \ref{LemUh} that 
			\begin{equation*}
				\norm{ S^{(1)}_{t} } \le C \norm{K_{z+tb}}_{*} \quad \text{ and } \quad \norm{S^{(2)}_{t}} \le C \norm{ \frac{1}{t}(K_{z+tb;a}-K_{z;a}) -K_{z;\alpha} }_{*}
			\end{equation*}
			Since \(a + b = \alpha\) with \(|b|=1\) and \(|a|=k\), it follows from the limit
			\begin{equation*}
				\lim_{ t \to 0}\left( \frac{K_{z+tb;a}(\zeta)-K_{z;a}(\zeta)}{t}  \right)= \lim_{ t \to 0}\left( \zeta^{a} e^{\langle \zeta , z \rangle} \frac{ e^{t \langle \zeta, b \rangle}-1 }{t}\right) = \zeta^{a} e^{\langle \zeta , z \rangle} \zeta^b = K_{z;\alpha}(\zeta)~,
			\end{equation*}
			that \( \lim_{t \to 0} \norm{\frac{1}{t}(K_{z+tb;a}-K_{z;a}) -K_{z;\alpha}}_{*}=0.\) Also,
			\begin{equation*}
				\lim_{t \to 0} \norm{K_{z+tb}}_{*} = \lim_{t \to 0} (2\pi)^{n/2} e^{|z+tb|^2}  = (2\pi)^{n/2} e^{|z|^2}<\infty ~ .
			\end{equation*}
			Therefore,
			\begin{equation*}
				\norm{S_{t}} \le \norm{S^{(1)}_{t}} \norm{S^{(2)}_{t}} \le C \norm{K_{z+tb}}_{*} \norm{ \frac{1}{t}(K_{z+tb;a}-K_{z;a}) -K_{z;\alpha} }_{*} \underset{t\to 0 }{\longrightarrow} 0~.
			\end{equation*}
			For \(T_{t}\), we have the factorization \(T_{t} = T^{(1)}_{t} \left(T^{(2)}\right)^{*} \), where
			\begin{equation*}
				T^{(1)}_{t} = \sum_{u \in \Z^{2n}} c_{u} \{U_{u}(K_{z+tb} -K_{z}) \} \otimes e_{u} \quad \text{ and } \quad T^{(2)} = \sum_{u \in \Z^{2n}} (U_{u}K_{z;\alpha}) \otimes e_{u}~.
			\end{equation*}
			By Lemma \ref{LemUh}, \( \norm{T^{(1)}_{t}} \le C \norm{K_{z+tb}-K_{z}}_{*} \) and \(T^{(2)}\) is bounded. Since \( \lim_{t \to 0} \norm{K_{z+tb}-K_{z}}_{*}=0  \), it follows that \( \norm{T_{t}} \le \norm{T^{(2)}} \norm{T^{(1)}_{t}}  \underset{t\to 0 }{\longrightarrow} 0  \)~. This shows (\ref{HtV}).
			
			To prove (\ref{GtW}), note that 
			\begin{equation*}
				G_{t} - W = \sum_{u \in \Z^{2n}} c_{u} \{ U_{u} ( t^{-1} (K_{z+tb} -K_{z}) - K_{z;b} ) \} \otimes (U_{u} K_{z;a}) = Z_{t} T^{(2)*}~,
			\end{equation*}
			where 
			\begin{equation*}
				Z_{t} = \sum_{u \in \Z^{2n}} c_{u} \{ U_{u} ( t^{-1} (K_{z+tb} -K_{z}) - K_{z;b} ) \}\otimes e_{u}~.
			\end{equation*}
			From Lemma \ref{LemUh} , we have 
			\begin{equation*}
				\norm{Z_{t}} \le C \norm{ t^{-1} (K_{z+tb} -K_{z}) - K_{z;b} }_{*} \underset{t\to 0 }{\longrightarrow} 0 ~.
			\end{equation*}
			Hence \(\norm{G_{t} - W} \le \norm{T^{(2)}} \norm{ Z_{t}} \underset{t\to 0 }{\longrightarrow} 0  \). Thus we have completed the proof of (\ref{AtA0}).
			
			The proof of (\ref{BtB0}) uses essentially the same arguments as above. This finishes the proof of the proposition.
		\end{proof}
		
		\begin{pro}\label{ProSumT1}
			Let \(\{c_{u} : u \in \Z^{2n}\}\) be a bounded set of complex coefficients. Then for every \(w\in \C^{n}\) we have 
			\begin{equation*}
				\sum_{u \in \Z^{2n}} c_{u} k_{u} \otimes k_{u-w} \in \mathcal{T}^{(1)}~.
			\end{equation*}
		\end{pro}
		
		\begin{proof}
			For each \( \alpha \in \N^n \), we define on \(\C^n\) the following monomial function
			\begin{equation*}
				p_{\alpha}(\zeta) = \zeta^{\alpha}~.
			\end{equation*}
			For \(u \in \Z^{2n} \) and \(w \in \C^n \), we define
			\begin{equation*}
				d_{u}(w) = c_{u} e^{-iIm\langle u, w \rangle}.
			\end{equation*}
			For all \(\alpha \in\N^n \) and \(u \in \Z^{2n} \),  we have \(K_{0;\alpha}=p_{\alpha}\) and \(U_{u}K_{0} =U_{u}\mathbbm{1}=k_{u}\). Thus, applying Proposition \ref{PropUuKzT1} in the case \(z=0\), we have that
			\begin{equation*}
				\sum_{u \in \Z^{2n}} c_{u}k_{u} \otimes (U_{u}p_{\alpha}) =\sum_{u \in \Z^{2n}} c_{u}(U_{u}K_{0}) \otimes (U_{u}K_{0;\alpha})  ~ \in \mathcal{T}^{(1)}~.
			\end{equation*}
			Hence 
			\begin{equation}\label{UuK0}
				\sum_{u \in \Z^{2n}} d_{u}(w)k_{u} \otimes (U_{u}p_{\alpha}) ~ \in \mathcal{T}^{(1)}~.
			\end{equation}
			We define the function \(g_{w}(\zeta) = \langle \zeta , w \rangle~, \zeta \in \C^n\) . For each \(j \in \N\), we define the operators 
			\begin{equation*}
				A_{j} = \sum_{u \in \Z^{2n}} d_{u}(w) k_{u} \otimes U_{u} g^{j}_{w} \quad \text{ and } \quad G = \sum_{u \in \Z^{2n}} (U_{u}K_{w}) \otimes e_{u}~.
			\end{equation*}
			Since each \(g^{j}_{w}\) is in the linear span of \(\{p_{\alpha} : \alpha \in \N^n\} \), then it follows from (\ref{UuK0}) that \(A_{j} \in \mathcal{T}^{(1)}\) for every \(j \in \N\). Also, for each \(j \in \N\), we have the factorization \(A_{j} = TB^{*}_{j} \), where
			\begin{equation*}
				T=\sum_{u \in \Z^{2n}} d_{u}(w) k_{u}\otimes e_{u} \quad \text{ and } \quad B_{j} = \sum_{u \in \Z^{2n}} (U_{u}g^{j}_{w})\otimes e_{u}~.
			\end{equation*}
			Hence, by Lemma \ref{LemUh}, the operator \(T\) is bounded and for each \(k \in \N \)
			\begin{eqnarray}\label{GBj}
				\norm{G-\sum_{j=0}^{k} \frac{1}{j!} B_{j} } &=& \norm{ \sum_{u \in \Z^{2n}} U_{u}\{ K_{w}-\sum_{j=0}^{k} \frac{1}{j!} g^{j}_{w} \} \otimes e_{u}  } \nonumber\\
				& \le& C \norm{ K_{w} - \sum_{j=0}^{k} \frac{1}{j!} g^{j}_{w}}_{*}~.
			\end{eqnarray}
			Using the expansion formula \( e^{c} =\displaystyle\sum_{j=0}^{\infty} \frac{1}{j!} c^j  \) for every \(c \in \C^n\,,\) we have
			\begin{equation*}
				\lim_{k \to \infty} \norm{K_{w} - \sum_{j=0}^{k} \frac{1}{j!} g^{j}_{w} }_{*} =0~.
			\end{equation*}
			Combining this with (\ref{GBj}), we obtain
			\begin{equation*}
				\lim_{k \to \infty} \norm{ TG^{*} - \sum_{j=0}^{k} \frac{1}{j!} A_{j} } \le \lim_{k \to \infty} \norm{T} \norm{G^{*}-\sum_{j=0}^{k} \frac{1}{j!} B^{*}_{j} } =0 ~.
			\end{equation*}
			Since each \(A_{j}\) belongs to \(\mathcal{T}^{(1)}\) and \(k_{w} = e^{-\frac{|w|^2}{2}}K_{w}\), we conclude that 
			\begin{equation*}
				\sum_{u \in \Z^{2n}} d_{u}(w) k_{u} \otimes (U_{u}K_{w}) = TG^{*} \in \mathcal{T}^{(1)} \quad \text{ and then } \quad \sum_{u \in \Z^{2n}} d_{u}(w) k_{u} \otimes (U_{u}k_{w}) \in \mathcal{T}^{(1)}~.
			\end{equation*}
			From the definition of \(d_{u}(w)\) and (\ref{Uu1}), we have that 
			\begin{equation*}
				\sum_{u \in \Z^{2n}} d_{u}(w) k_{u} \otimes (U_{u}k_{w})=\sum_{u \in \Z^{2n}} c_{u} k_{u} \otimes k_{u-w}.
			\end{equation*}
			This completes the proof.
		\end{proof}
		
		\begin{proof}[\textbf{Proof of Proposition \ref{DoT1}} ]
			Let \(\{c_{u}: u \in \Z^{2n}\}\) be a bounded set of coefficients and \(\gamma: \Z^{2n} \rightarrow \C^n\) a map for which there exists \(0<C< \infty\) such that \(\norm{u-\gamma(u)} \le C \) for every \(u\in \Z^{2n}\). Let \(\mathcal{K}=\{w\in \C^{n}: \norm{w}\le C\}\). We want to show that the operator
			\begin{equation*}
				T = \sum_{u \in \Z^{2n}} c_{u} k_{u} \otimes k_{\gamma(u)}
			\end{equation*}
			belongs to \(\mathcal{T}^{(1)}\). For this we define
			\begin{equation*}
				\psi(u) = u-\gamma(u)~, \quad u \in \Z^{2n}.
			\end{equation*}
			Then \(\psi(u) \in \mathcal{K}\) for every \(u\in \Z^{2n}\) and \(\varphi_{u}(\psi(u)) = \gamma(u)\). By (\ref{Uu1}), we have \(U_{u}k_{\psi(u)} = k_{\gamma(u)} e^{iIm\langle u, \psi(u) \rangle}\). Therefore
			\begin{eqnarray*}
				T &=& \sum_{u \in \Z^{2n}} c_{u}k_{u} \otimes \left(U_{u}k_{\psi(u)} e^{-iIm\langle u,\psi(u) \rangle}\right) =\sum_{u \in \Z^{2n}} d_{u} k_{u} \otimes (U_{u}k_{\psi(u)}) \, ,
			\end{eqnarray*}
			where \(d_{u} = c_{u}  e^{iIm\langle u,\psi(u) \rangle} \) for every \(u\in \Z^{2n}\) and we have \(|d_{u}|=|c_{u}|\). Then the operator \(T\) can be factorized as follows \(T=AB^{*}\), where
			\begin{equation*}
				A = \sum_{u \in \Z^{2n}} d_{u} k_{u} \otimes e_{u} \quad \text{ and } \quad B = \sum_{u \in \Z^{2n}} (U_{u}k_{\psi(u)}) \otimes e_{u}~. 
			\end{equation*}
			Since the map \(z \mapsto k_{z}\) is \(\norm{\cdot}_{*}\)-- continuous ( that is \(\lim_{w\to z}\norm{k_w - k_{z}}_{*} =0 \)) , and \(\mathcal{K}\) is compact, then it is uniformly continuous on \(\mathcal{K}\). Therefore, for a given \(\varepsilon>0\), the compactness of \(\mathcal{K}\) implies that there are non-empty open sets \(\Omega_1,\dots,\Omega_{m}\) and \(z_{i} \in \Omega_{i}\), \(i\in \{1,\dots,m\}\), such that 
			\begin{equation*}
				\Omega_{1} \cup \dots \cup \Omega_{m} \supset \mathcal{K}  \quad \text{ and } \quad  \norm{k_{z_i}-k_{w}}_{*} < \varepsilon \quad \text{ whenever } \quad w \in \Omega_{i}, i \in \{1,\dots,m\}.
			\end{equation*}
			From that open cover of \(\mathcal{K}\), we obtain a partition \(\mathcal{K}= E_{1} \cup \cdots \cup E_{m}\) such that \(E_{i} \subset \Omega_{i}\) for every \(i\in \{1,\dots,m\}\). We now define \(\Gamma_{i}=\{u\in \Z^{2n}: \psi(u) \in E_{i} \}\), \(i \in \{1,\dots,m\}\). Then \(\norm{k_{z_{i}} - k_{\psi(u)}}_{*}< \varepsilon \) if \(u \in \Gamma_{i}\). For all \(i \in \{1,\dots,m\} \), we also define 
			\begin{equation*}
				B_{i} = \sum_{u\in \Gamma_{i}} (U_{u}k_{z_i}) \otimes e_{u}~.
			\end{equation*}
			Then for each \( i \in \{1,\dots,m\}\), and by (\ref{Uu1})  , we have 
			\begin{equation*}
				AB^{*}_{i}  = \sum_{u\in \Gamma_{i}} d_{u} e^{-iIm\langle u,z_{i} \rangle } k_{u} \otimes k_{u-z_{i}} = \sum_{u\in \Gamma_{i}} d_{u,i}  k_{u} \otimes k_{u-z_{i}} ~,
			\end{equation*}
			where \(d_{u,i} =d_{u} e^{-iIm\langle u,z_{i}\rangle} \) and \(|d_{u,i}|=|d_{u}| = |c_{u}|\) for \(u \in \Gamma_{i} \). Thus, it follows from Proposition \ref{ProSumT1} that 
			\begin{equation}\label{ABT1}
				\{AB^{*}_{1}, \dots, AB^{*}_{m}\} \subset \mathcal{T}^{(1)}.
			\end{equation}
			On the other hand, we have 
			\begin{equation*}
				B - (B_{1} + \cdots + B_{m}) = \sum_{i=1}^{m} \sum_{u\in \Gamma_{i}} \{ U_{u}(k_{\psi(u)}-k_{z_{i}}) \} \otimes e_{u}~.
			\end{equation*}
			It follows from the fact that \(\Gamma_{1},\dots, \Gamma_{m}\) form a partition of \(\Z^{2n}\) and from Lemma \ref{LemUh} that 
			\begin{equation*}
				\norm{ B - (B_{1} + \cdots + B_{m}) } \le C \max_{1\le i \le m} \sup_{u\in \Gamma_{i}} \norm{ k_{\psi(u)} -k_{z_{i}} }_{*} \le C \varepsilon~.
			\end{equation*}
			We also have from Lemma \ref{LemUh} that \(A\) is a bounded operator. Hence
			\begin{eqnarray*}
				\norm{T - (AB^{*}_{1} + \cdots + AB^{*}_{m})} &=& \norm{AB^{*} - (AB^{*}_{1} + \cdots + AB^{*}_{m})} \\
				&\le& \norm{A} \norm{ B^{*} - (B^{*}_{1} + \cdots + B^{*}_{m}) } \\
				&=& \norm{A} \norm{ B - (B_{1} + \cdots + B_{m}) } \\
				&\le& C \norm{A} \varepsilon~.
			\end{eqnarray*}
			Since this inequality holds for an arbitrary \(\varepsilon>0\), then combined with (\ref{ABT1}) , we conclude that \(T\in \mathcal{T}^{(1)}\). This completes the proof.
		\end{proof}
		
		\section{Proof and a Consequence of the Main Result }
		
		In this section, we establish the proof of Theorem \ref{XiaFock} and use the result of Bauer and Isralowitz in \cite{wangzhu2013boundedness} to deduce some of its consequences.
		
		\begin{proof}[\textbf{Proof of Theorem \ref{XiaFock} }]
			Since \(\mathcal{WL}\) is a \(*\)-algebra, \(C^{*}(\mathcal{WL})\) is just the norm closure of \(\mathcal{WL}\). Also, from Example \ref{Ex1}, we know that \(\mathcal{WL} \supset \{T_{f}: f \in L^{\infty}(\C^n, dV) \}\). Therefore \(\mathcal{T}^{(1)} \subset C^{*}(\mathcal{WL})\). Hence we just have to show that \(\mathcal{WL} \subset \mathcal{T}^{(1)}\) to complete the proof.
			
			Let \(B \in \mathcal{WL} \), then from (\ref{Bint}), we have 
			\begin{equation}\label{Bint2}
				B = \int_{S} \int_{S} E_{w} B E_{z} dV(w) dV(z)~. 
			\end{equation}
			From Proposition \ref{BT1} we know  that, the range of the map
			\begin{equation}\label{mapEwzBEz}
				(w,z) \longmapsto E_{w} B E_{z},
			\end{equation}
			defined from \(\C^n \times \C^n\) into \(\mathcal{B}(H^{2}(\C^n,d\mu))\) is contained in \(\mathcal{T}^{(1)} \). Therefore, every Riemann sum corresponding to the integral defined by the relation (\ref{Bint2}) belongs to \(\mathcal{T}^{(1)} \). Moreover, from Proposition \ref{Yz},  we know that the map \(z\mapsto E_{z}\) is continuous from \(\C^n\) into \(\mathcal{B}(H^{2}(\C^n,d\mu))\) with respect to the operator norm. Since the closure of \(S\times S\) is a compact subset of \(\C^n \times \C^n \), the norm continuity of (\ref{mapEwzBEz}) means that the integral in (\ref{Bint2}) is the limit with respect to the operator norm of a sequence of Riemann sums \(s_{1},\dots,s_{k}\). Hence, the fact that each \(s_{k}\) belongs to \(\mathcal{T}^{(1)} \), implies that \(B\) belongs to \(\mathcal{T}^{(1)} \).
		\end{proof}
		
		In the continuation of our work, we define the notion of Berezin transform and we state a corollary of Theorem \ref{XiaFock} which highlights its consequence on the compactness of bounded linear operators on the Fock space \(H^{2}(\C^n , d\mu)\). Foremost we recall the definition of a compact operator on \(H^{2}(\C^n, d\mu)\).
		
		\begin{defn}
			A bounded linear operator \(T\) on \(H^{2}(\C^n,d\mu)\) is \textbf{compact} if for every sequence \(\{f_{n}\}_n\) of elements of \(H^{2}(\C^n, d\mu)\)  converging weakly to zero in \(H^{2}(\C^n, d\mu)\), the sequence \(\{Tf_{n}\}_n\) converges to zero with respect to the norm topology of \(H^{2}(\C^n, d\mu)\).
		\end{defn}
		
		\begin{defn}
			Let \(A\) be a bounded linear operator on \(H^{2}(\C^n , d\mu)\). The \textbf{Berezin transform} of \(A\) is the function denoted by \(B(A)\) and defined by 
			\begin{equation*}
				B(A)(z) = \langle Ak_{z}, k_{z} \rangle ~,\qquad z \in \C^n~. 
			\end{equation*}
		\end{defn}
		
		\begin{rem}
			We can reformulate the result of Bauer and Isralowitz \cite[Theorem 1.1]{bauer2012compactness} as follows: A bounded linear operator on \(H^{2}(\C^n, d\mu)\) is compact if and only if it belongs to the \(C^{*}\)-algebra generated by weakly localized operators on \(H^{2}(\C^n, d\mu)\) and its Berezin transform vanishes at infinity.
		\end{rem}
		
		\begin{cor}\label{CorXia}
			A weakly localized operator is compact if and only if its Berezin transform vanishes at infinity.
		\end{cor}
		\begin{proof}
			The proof follows directly from \cite[Theorem 1.1]{bauer2012compactness} and Theorem \ref{XiaFock}.
		\end{proof}

	\section*{Acknowledgments}
	
	The author wishes to thank AIMS Cameroon (the African Institute for Mathematical Sciences) for giving her the opportunity to explore this topic during her Master degree.

	\renewcommand{\bibname}{References}
	\bibliographystyle{plain}
	\bibliography{template-bibliography}

\begin{thebibliography}{1}

\bibitem{bauer2012compactness}
W.~Bauer and J.~Isralowitz.
\newblock Compactness characterization of operators in the {Toeplitz} algebra
  of the {Fock} space f$\alpha$p.
\newblock {\em Journal of Functional Analysis}, 263(5):1323--1355, 2012.

\bibitem{wangzhu2013boundedness}
X.~Wang, G.~Cao, and K.~Zhu.
\newblock Boundedness and {Compactness} of {Operators} on the {Fock} {Space}.
\newblock {\em Integr. Equ. Oper. Theory}, 77:355--370, 2013.

\bibitem{xia2015localization}
J.~Xia.
\newblock Localization and the {Toeplitz} algebra on the {Bergman} space.
\newblock {\em Journal of Functional Analysis}, 269(3):781--814, 2015.

\bibitem{xia2013localization}
J.~Xia and D.~Zheng.
\newblock Localization and {Berezin} transform on the {Fock} space.
\newblock {\em Journal of Functional Analysis}, 264(1):97--117, 2013.

\bibitem{zhu2007}
K.~Zhu.
\newblock {\em Operator Theory in Function Spaces}, volume 138 of {\em
  Mathematical Surveys and Monographs}.
\newblock American Mathematical Society, Providence, Rhode Island, USA, 2007.

\bibitem{zhu2012}
K.~Zhu.
\newblock {\em Analysis on {Fock} spaces}, volume 263 of {\em Graduate texts in
  Mathematics}.
\newblock Springer, New York, 2012.

\end{thebibliography}
	\addcontentsline{toc}{chapter}{References}
	
	\vspace{2cm}
	AIMS-CAMEROON, CRYSTAL GARDENS, P.O.Box 608, LIMBE,CAMEROON\\
	Email address: solange.difo@aims.cameroon.org
	
\end{document}